\documentclass{IEEEtran}

\usepackage{amssymb}
\usepackage{amsthm}
\usepackage{cite}
\usepackage[multiple]{footmisc}
\usepackage{mathtools}
\usepackage{enumitem}
\usepackage{color}
\mathtoolsset{showonlyrefs}

\newcommand{\e}{\varepsilon}

\newcommand{\ds}{\, ds}

\newcommand{\R}{\mathbb{R}}

\newcommand{\grad}{\nabla}

\newcommand{\vyavg}{\vy_{\textup{avg}}}

\newcommand{\ones}{{\bf 1}}

\newcommand{\ddt}{\frac{d}{dt}}
\newcommand{\cl}{\textup{cl}}
\newcommand{\dtau}{\,d\tau}

 \def\vr{{\bf r}}  
\def\vu{{\bf u}}   \def\vx{{\bf x}}
\def\vy{{\bf y}} \def\vz{{\bf z}}

  \def\calC{\mathcal{C}}

 \def\calN{\mathcal{N}} 
  
\def\calS{\mathcal{S}} \def\calT{\mathcal{T}}

\newtheorem{theorem}{Theorem}
\newtheorem{assumption}{Assumption}

\newtheorem{lemma}[theorem]{Lemma}
\newtheorem{definition}[theorem]{Definition}

\title{Distributed Gradient Descent: Nonconvergence to Saddle Points and the Stable-Manifold Theorem}
\author{Brian Swenson$^\dagger$, Ryan Murray$^\star$, H. Vincent Poor$^\dagger$, and Soummya Kar$^\ddagger$
\thanks{\noindent This work was partially supported by the Air Force Office of Scientific Research under MURI Grant FA9550-18-1-0502. %and was partially supported by the National Science Foundation under Award Number CCF 1513936
\newline
$^\dagger$Department of Electrical Engineering, Princeton University, Princeton, NJ 08540 (bswenson@princeton.edu and poor@princeton.edu),\newline
$^\star$Department of Mathematics, North Carolina State University, Raleigh, NC 27695 (rwmurray@ncsu.edu) \newline
$^\ddagger$Department of Electrical and Computer Engineering, Carnegie Mellon University, Pittsburgh, PA 15213 (soummyak@andrew.cmu.edu)}
}

\begin{document}

\maketitle

% REQUIRED
\begin{abstract}
The paper studies continuous-time distributed gradient descent (DGD) and considers the problem of showing that in nonconvex optimization problems, DGD typically converges to local minima rather than saddle points. In centralized settings, the problem of demonstrating nonconvergence to saddle points
is typically handled by way of the stable-manifold theorem from classical dynamical systems theory. However, the classical stable-manifold theorem is not applicable in the distributed setting. The paper develops an appropriate stable-manifold theorem for DGD. This shows that convergence to saddle points may only occur from a low-dimensional stable manifold.
Under appropriate assumptions (e.g., coercivity), the result implies that DGD almost always converges to local minima.
\end{abstract}

% REQUIRED

\section{Introduction}
Suppose a group of $N$ agents may communicate over a network. Each agent possesses some local function $f_n:\R^d\to\R$ and it is desired to optimize the sum function $f:\R^d\to\R$ given by
\begin{equation} \label{eqn:f-def}
f(x):= \sum_{n=1}^N f_n(x).
\end{equation}
In applications, the function $f_n$ is typically generated from local information available only to agent $n$, and \eqref{eqn:f-def} represents some collective objective a system designer would like to optimize
\cite{rabbat2004distributed,chen2012diffusion,cao2012overview,boyd2011distributed,kar2019clustering}.
We are interested in the use of distributed gradient descent processes to compute local optima of \eqref{eqn:f-def} wherein agents may only exchange information with neighboring agents
%in the network and there is no central entity (or node) to coordinate the computation,
e.g., \cite{nedic2009distributed}.

In this paper we focus on the case where the local $f_n$ functions may be \emph{nonconvex}.
This framework encompasses a wide range of applications including, for example, empirical risk minimization \cite{lee2018distributed}, target localization \cite{di2015distributed}, robust regression \cite{sun2016distributed}, distributed coverage control \cite{welikala2019distributed}, power allocation in wireless adhoc networks \cite{bianchi2012convergence}, and others \cite{di2016next}.

Assuming the objective is smooth, basic convergence results in nonconvex optimization typically ensure that algorithms converge to critical points. This set consists, of course, of local and global minima and saddle points.
Global minima can be difficult to compute and, for practical purposes, local minima are often sufficient in applications \cite{dauphin2014identifying}. Thus, global optima aside, the main difficulty in proving that an algorithm has desirable convergence properties typically lies in understanding the behavior near saddle points, and, in particular, showing nonconvergence to saddle points \cite{lee2016gradient,lee2017first,murray2019revisiting}.

For classical (centralized) gradient descent, the problem of showing non-convergence to saddle points is handled using the well-known ``stable-manifold theorem'' from dynamical systems theory \cite{shub2013global,coddington1955theory,lee2016gradient}.
In short, the stable-manifold theorem says that gradient descent (along with many other first-order algorithms \cite{lee2017first}) can only converge to a saddle point if initialized on some low-dimensional hypersurface (referred to as the stable manifold).\footnote{The stable-manifold theorem deals with unstable points of general dynamical systems, not just gradient-type systems. However, restricted to gradient-type systems, this is the main implication of the result.} Any process initialized on the stable manifold will remain on the stable manifold thereafter, eventually converging to the saddle point of interest. On the other hand, any process not initialized on the stable manifold will be repelled from the saddle point (eventually converging to some local minimum, assuming, for example, that $f$ is coercive). In this way, the problem of understanding (non)convergence to saddle points in classical settings is completely resolved by the stable-manifold theorem.

In the distributed setting, this is not the case. The classical stable-manifold theorem does not generally apply and specialized stable-manifold theorem results do not exist. Several recent works, including \cite{bianchi2012convergence,di2016next,sun2016distributed,swenson2019CDC,tatarenko2017non}, have considered gradient-descent type algorithms for distributed nonconvex optimization. These have shown convergence to critical points, but have not dealt with the issue of nonconvergence to saddle points. The recent work \cite{daneshmand2018second} considered discrete-time distributed gradient descent with constant step size and demonstrated convergence to a neighborhood of a second-order stationary point under relatively mild assumptions.

In this work we focus on continuous-time dynamics and consider the problem of characterizing the stable manifold for the distributed gradient descent process
\begin{equation} \label{dynamics_CT}
\dot \vx_n(t) = \beta_t \sum_{\ell \in \Omega_n} (\vx_\ell(t) -\vx_n(t)) - \alpha_t \grad f_n(\vx_n(t)),
\end{equation}
$n=1,\ldots,N$, where $\alpha_t$, and $\beta_t$ are time-varying (decaying) weight parameters, and $\Omega_n$ is the set of agents neighboring agent $n$ in the underlying communication graph. Intuitively, the dynamics \eqref{dynamics_CT} may be understood as follows: The \emph{consensus} term $\beta_t\sum_{\ell \in \Omega_n} (\vx_\ell(t) -\vx_n(t))$ encourages agents to seek agreement with neighboring agents. The \emph{innovation} term $-\alpha_t \grad f_n(\vx_n(t))$ encourages each agent to descend the gradient of their local objective function. By appropriately controlling the decay rates of $\alpha_t$ and $\beta_t$ one can balance the dual objectives of ensuring that agents reach asymptotic consensus while simultaneously seeking optima of \eqref{eqn:f-def}. The process \eqref{dynamics_CT} is a consensus + innovations variant of gradient descent \cite{kar2012distributed}.

We remark that closely related discrete-time variants of distributed gradient descent were studied in \cite{nedic2009distributed,nedic2010constrained,nedic2014distributed} for distributed optimization of a convex function. This was extended to the distributed nonconvex setting in \cite{bianchi2012convergence} where convergence to critical points was shown.
The work \cite{swenson2019CDC} considered a distributed simulated annealing algorithm that ensures convergence to the set of global minima. However, the algorithm requires careful control of the annealing noise.
We also remark that the recent work \cite{hong2018gradient} considered a discrete-time primal dual algorithm for distributed nonconvex optimization and showed convergence to second-order stationary points, but did not consider distributed gradient descent.

Our first main result will be to show that the dynamics \eqref{dynamics_CT} converge to critical points of $f$ (see Theorem \ref{thrm:cont-conv-cp}).
Our second main result will be to prove a stable-manifold theorem for \eqref{dynamics_CT} that characterizes nonconvergence to saddle points (see Theorem \ref{thrm:non-convergence}).
Together, these results show that (under appropriate assumptions) the dynamics \eqref{dynamics_CT} typically converge to local minima of \eqref{eqn:f-def}.

\subsection{Main Results}
\subsubsection{Assumptions}
We will make the following general assumptions.

The first assumption pertains to the communication network.
\begin{assumption} \label{a:G-connected-undirected}
The graph $G=(V,E)$ is undirected and connected.
\end{assumption}
(See Section \ref{sec:notation} for further discussion of the communication network.)
The next three assumptions apply to the local objectives $f_n$, $n=1,\ldots,N$.
\begin{assumption} \label{a:f-differentiable}
$f_n:\R^d\to \R$ is of class $C^2$.
\end{assumption}
\begin{assumption} \label{a:lip-grad}
$\grad f_n$ is Lipschitz continuous,
\end{assumption}
\begin{assumption} \label{a:coercive}
$f_n$ is coercive.
\end{assumption}
We refer to the time-varying weights $\beta_t$ and $\alpha_t$ in \eqref{dynamics_CT} as the \emph{consensus} and \emph{innovation} potentials respectively.
We assume the consensus and innovation potentials take the following form.
\begin{assumption} \label{a:step-size-CT}
$\alpha_t = (t+1)^{-\tau_\alpha}$ and $\beta_t = (t+1)^{-\tau_\beta}$, with $0\leq \tau_\beta < \tau_\alpha \leq 1$.
\end{assumption}

When developing our stable-manifold theorem for \eqref{dynamics_CT} we will consider the behavior of the dynamics near some fixed saddle point $x^*$. We will assume that the saddle point satisfies the following non-degeneracy assumption.
\begin{assumption} \label{a:saddle-point}
$x^*$ is a nondegenerate saddle point of $f$. That is, the Hessian $\nabla^2 f(x^*)$ is nonsingular.
\end{assumption}

\subsubsection{Main Results}
We now state the main results of the paper. First, we show that the dynamics \eqref{dynamics_CT} converge to the set of critical points of \eqref{eqn:f-def}.
%Our first main result shows that the dynamics \eqref{dynamics_CT} converge to critical points of \eqref{eqn:f-def}.
\begin{theorem}\label{thrm:cont-conv-cp}
Suppose $(\vx_n(t))_{n=1}^N$ is a solution to \eqref{dynamics_CT} with arbitrary initial condition and suppose that Assumptions \ref{a:G-connected-undirected}--\ref{a:step-size-CT} hold.  Then for each $n=1,\ldots,N$,
\begin{itemize}
  \item [(i)]  Agents achieve consensus in the sense that $\lim_{t\to\infty} \|\vx_n(t) - \vx_\ell(t)\| = 0$, for $\ell=1,\ldots,N$.
  \item [(ii)] $\vx_n(t)$ converges to the set of critical points of $f$.
\end{itemize}
\end{theorem}

Our second main result will refine this convergence guarantee.  The next result shows that the critical point reached by \eqref{dynamics_CT} will not typically be a saddle point. We show the following stable-manifold theorem for \eqref{dynamics_CT}.
\begin{theorem} \label{thrm:non-convergence}
Suppose that Assumptions \ref{a:G-connected-undirected}--\ref{a:step-size-CT} hold and suppose that $x^*$ is a saddle point of $f$ satisfying Assumption \ref{a:saddle-point}. Let $p$ denote the number of negative eigenvalues of the Hessian $\nabla^2 f(x^*)$. Then for all $t_0$ sufficiently large there exist a manifold $S\subset \R^{Nd}$ with dimension $(Nd-p)$ such that the following holds: A solution $(\vx_n(t))_{n=1}^N$ to \eqref{dynamics_CT} converges to $x^*$ in the sense that $\vx_n(t)\to x^*$ for some $n$, if and only if $(\vx_n(t))_{n=1}^N$ is initialized on $S$, i.e., $(\vx_n(t_0))_{n=1}^N = x_0\in\R^{Nd}$ with $x_0\in S$.
\end{theorem}
When we say that $S$ has dimension $Nd-p$ we mean that $S$ is the graph of a continuous function from a $Nd-p$ dimensional domain.
Note that in the above theorem, since we deal with a nondegenerate saddle point of $f$, we must have $p\geq 1$. Thus, $S$ has dimension at most $Nd-1$ and is indeed a ``low-dimensional surface.'' The initial time $t_0$ in the above theorem depends on the weight processes $\alpha_t$ and $\beta_t$. This time may be equivalently taken to be zero by using alternate weight sequences $\hat \alpha_t = \alpha_{t+t_0}$ and $\hat \beta_t = \beta_{t+t_0}$.

The value of Theorems \ref{thrm:cont-conv-cp} and \ref{thrm:non-convergence} together are that they allow us to conclude that the dynamics \eqref{dynamics_CT} ``typically'' converge to local minima of $f$ (assuming Assumptions \ref{a:G-connected-undirected}--\ref{a:step-size-CT} hold and every saddle point of $f$ satisfies Assumption \ref{a:saddle-point}). More precisely,
Theorem \ref{thrm:cont-conv-cp} tells us  that the dynamics \eqref{dynamics_CT} will converge to critical points of $f$. Theorem \ref{thrm:non-convergence} tells us that this limit point \emph{must} be a local minimum\footnote{In the event that $\vx(t)$ does not have a unique limit, then it converges to a connected set of local minima.} unless $(\vx_n(t))_{n=1}^N$ is initialized from the special set of initial conditions $\bigcup_{x^*} S_{x^*}$, where the (countable) union is taken over the set of all saddle points, and each $S_{x^*}$ is the low-dimensional stable manifold associated with the saddle point $x^*$.

It is also important to remark that a shortcoming of Theorem \ref{thrm:non-convergence} is that it does not show that $S$ is a smooth $C^1$ surface. This will be the subject of future work.

The remainder of the paper is organized as follows. Section \ref{sec:notation} sets up notation and reviews background material. %Section \ref{section:prelim-perturbed-solns} reviews background material in perturbed solutions of ODEs.
Section \ref{sec:conv-to-cp} proves Theorem \ref{thrm:cont-conv-cp}. Section \ref{sec:nonconvergence} proves Theorem \ref{thrm:non-convergence}. Finally, Section \ref{sec:conclusion} concludes the paper.

\section{Notation} \label{sec:notation}
Let $C^k(\R^{d_1}; \R^{d_2})$ denote the set of all $k$-times continuously differentiable functions from $\R^{d_1}$ to $\R^{d_2}$. When the dimensions of domain and codomain are clear, we will simply say that a function belongs to $C^k$.
Given a function $f\in C^2$, we let $\nabla f(x)$ denote the gradient of $f$ and let $\nabla^2 f(x)$ denote the Hessian.
Unless otherwise stated, $\|\cdot\|$ refers to the standard Euclidean norm. Given a point $x\in\R^m$ and $r>0$ let $B_r(x)$ denote the open ball of radius $r>0$ about $x$. We use the notation $I_m$ to denote the $m\times m$ identity matrix. Given a matrix $A$, $\calN(A)$ denotes the nullspace of $A$.  Given a set of numbers $\{a_1,\ldots,a_m\}$ let $\text{diag}(a_1,\ldots,a_m)$ be the $m\times m$ diagonal matrix with diagonal entries $a_1,\ldots,a_m$.

We say that a continuous mapping $\vx:I \to \R^d$, over some interval $I =[0,T)$, $0 < T \leq \infty$, is a solution to an ODE with initial condition $x_0$ at time $t_0$ if $\vx \in C^1$, $\vx$ satisfies the ODE for all $t\in I$, and $\vx(t_0) = x_0$. We note that under Assumption \ref{a:lip-grad}, solutions to \eqref{dynamics_CT} exist and are unique \cite{coddington1955theory}.

In Assumption \ref{a:G-connected-undirected} we assume that the inter-agent communication graph may be described by an undirected graph $G=(V,E)$, where $V=\{1\cdots N\}$ denotes the set of nodes (or agents) and~$E$ denotes the set of communication links (edges), between agents. The pair $(n,l)\in E$ if and only if there exists an edge between nodes~$n$ and~$l$. In this paper we will consider simple graphs, i.e., graphs devoid of self-loops and multiple edges. The set of neighbors of node~$n$ is given by
\begin{equation}
\label{def:omega} \Omega_{n}=\left\{l\in V\,|\,(n,l)\in
E\right\}. %n\in \left[1\cdots N\right]
\end{equation}
The degree of node~$n$ is given by $d_{n}=|\Omega_{n}|$.
The adjacency matrix of the graph $G$ is the $N\times N$ matrix $A=\left[A_{nl}\right]$, with $A_{nl}=1$, if $(n,l)\in E$, $A_{nl}=0$, otherwise. The degree matrix is given by the diagonal matrix $D=\mbox{diag}\left(d_{1},\ldots,d_{N}\right)$. The positive semidefinite matrix $L=D-A$ is referred to as the graph Laplacian matrix. The eigenvalues of $L$ can be ordered as $0=\lambda_{1}(L)\leq\lambda_{2}(L)\leq\cdots\leq\lambda_{N}(L)$.
%The multiplicity of the zero eigenvalue equals the number of connected components of the network.
A graph is said to be connected if there exists a path between each pair of nodes.
If the graph $G$ is connected then $\lambda_{2}(L)>0$ \cite{chung1997spectral}.

\subsection{Stochastic Approximation and Perturbed Solutions} \label{section:prelim-perturbed-solns}
Some of our proof techniques will utilize results on perturbed solutions to differential equations from the theory of stochastic approximation. We briefly review relevant results from the literature now.

We will be interested in studying (possibly perturbed) solutions of the differential equation
\begin{equation} \label{eqn:generic-ode}
\dot \vx = F(\vx),
\end{equation}
where $F:\R^d\to \R$ is $C^1$.
We will consider the following notion of a perturbed solution.
\begin{definition} [Perturbed Solution] \label{def:perturbed-soln}
A continuous function $\vy:[0,\infty)\to \R^m$ will be called a perturbed solution to \eqref{eqn:generic-ode} if:\\
\begin{enumerate}
    \item $\vy$ is absolutely continuous,
    \item There exists a locally integrable function $t\mapsto U(t)$ such that for every $T>0$ there holds
    \begin{enumerate}
        \item
        $$
        \lim_{t\to\infty} \sup_{0\leq v\leq T} \Big| \int_t^{t+v} U(s)\ds \Big| = 0
        $$
        \item
        $$
        \ddt \vy(t) - U(t) = F(\vy(t))
        $$
        for almost every $t>0$.
    \end{enumerate}
\end{enumerate}
\end{definition}

Let $\Lambda \subset \R^d$; we say that a continuous function $V:\R^d\to \R$ is a Lyapunov function for $\Lambda$ if for any solution $\vx:\R\to \R^d$ of  \eqref{eqn:generic-ode} , $\ddt V(\vx(t))=0$ for $\vx(t) \in \Lambda$ and $\ddt V(\vx(t))<0$ for $\vx(t)\notin \Lambda$.

The following result (see Theorem 3.6 and Proposition 3.27 in \cite{benaim2005stochastic}) characterizes the asymptotic behavior of perturbed solutions to ODEs admitting a Lyapunov function.
\begin{theorem} \label{thrm:BPS-limit}
Suppose $\vy$ is a perturbed solution to \eqref{eqn:generic-ode}. Suppose also that $V$ is a Lyapunov function for $\Lambda$ and that $V(\Lambda)$ has empty interior. Then the limit set of $\vy$, given by $L(\vy) := \cap_{t\geq 0} \cl(\{\vy(s):s\geq t\})$ is contained in $\Lambda$.
\end{theorem}

\section{Convergence to Critical Points} \label{sec:conv-to-cp}
In this section we will prove Theorem \ref{thrm:cont-conv-cp}.
We begin by showing the following preliminary lemma which shows that under the dynamics \eqref{dynamics_CT} agents reach asymptotic consensus.
\begin{lemma} \label{lemma:cont-consensus}
If $(\vx_n(t))_{n=1}^N$ is a solution to \eqref{dynamics_CT} then $\lim_{t\to\infty} \|\vx_n(t) - \vx_\ell(t)\| =0$ for all $\ell,n=1,\ldots,N$.
\end{lemma}

\begin{proof}
%Let $L\in \R^{N\times N}$ denote the graph Laplacian associated with $G$.
The dynamics \eqref{dynamics_CT} may be expressed compactly as
\begin{equation} \label{eqn:non-homogenous-dynamics0}
\dot \vx = -\beta_t(L\otimes I_d) \vx - \alpha_t(\grad f_n(\vx))_{n=1}^N,
\end{equation}
where $\alpha_t$ and $\beta_t$ are as in Assumption \ref{a:step-size-CT}.

Let $S(\tau) = \int_{0}^\tau \beta_r \,dr$ and let $T(t)$ denote the inverse of $S(\tau)$. Let $\vy(t) = \vx(T(t))$. Using this time change we have the equivalent ODE
\begin{equation} \label{eqn:non-homogenous-dynamics}
\dot \vy(t) = -(L\otimes I_N)\vy(t) - \gamma_t(\nabla f_n(\vy(t)))_{n=1}^N,
\end{equation}
where $\gamma_t = \frac{\alpha_{T(t)}}{\beta_{T(t)}} \to 0$ as $\tau\to\infty$. Using the explicit form of $\alpha_t$ and $\beta_t$ in Assumption \ref{a:step-size-CT} it is readily verified that $\gamma_t \leq (t+1)^{-\tau_\gamma}$ for some $\tau_\gamma>0$.

We will refer to the set
$$
\calC := \{x\in \R^{Nd}: x = \ones_N \otimes a,\text{ for some } a\in \R^d\}
$$
as the \emph{consensus subspace}.
Consider the linear system
\begin{equation}\label{eqn:homogenous-dynamics}
\dot \vy = -(L\otimes I_d)\vy.
\end{equation}
Because $(L\otimes I_d)$ is positive semidefinite with nullspace equal to $\calC$, solutions to \eqref{eqn:homogenous-dynamics} converge to $\calC$ and hence $\lim_{t\to\infty} \|\vy_n(t) - \vy_\ell(t)\| = 0$ for all $n,\ell=1,\ldots,N$.

Let $\Phi(t) = e^{-(L\otimes I_d)t}$ denote a fundamental matrix solution of the linear system \eqref{eqn:homogenous-dynamics}.  By variation of parameters \cite{lakshmikantham1969differential}, the solution $\vy(t)$ of \eqref{eqn:non-homogenous-dynamics}
with initial condition $x_0\in \R^{Nd}$ may be expressed as
\begin{equation} \label{eqn:var-of-params}
\vy(t) = \Phi(t)x_0 + \int_{0}^t \Phi(t-s)b(s)\ds,
\end{equation}
where $b(s) = -\gamma_s(\grad f_n(\vy_n(s)))_{n=1}^N$. Using Assumptions \ref{a:lip-grad} and \ref{a:coercive} we see that $\|b(s)\| \leq \gamma_sC$ for some constant $C>0$.

Let
\begin{equation}
\vyavg(t) := \frac{1}{N} \sum_{n=1}^N \vy_n(t).
\end{equation}
Using  \eqref{eqn:var-of-params} we have
\begin{align}
\vy^\perp(t) & := \vy(t) - (\ones_N\otimes I_d)\vyavg(t)\\
& =\Phi(t)x_0 - (\ones_N\otimes I_d)\frac{1}{N} \sum_{n=1}^N \left[ \Phi(t)x_0\right]_n\\
&~~ + \int_0^t \Phi(t-s)\left( b(s) - (\ones_N\otimes I_d)\frac{1}{N}\sum_{n=1}^N \left[b(s) \right]_n \right),
\end{align}
where we have used the notation $[\cdot]_n$ to indicate extracting the vector of coordinates in $\R^d$ corresponding to agent $n$. Using the previous bound on $[b(t)]_n$ we get
\begin{align} \label{eq_consensus_conv}
\|\vy^\perp(t)\| \leq & \|\Phi(t)x_0 - \frac{1}{N} \sum_{n=1}^N \left[ \Phi(t)x_0\right]_n\|\\
& + C\int_0^t \|\Phi(t-s)\gamma_s\|\ds,
\end{align}
for some $C>0$.
The first term on the right hand side above goes to zero since $\Phi(t)x_0$ is a solution to \eqref{eqn:homogenous-dynamics}.
Recalling that $\Phi(t) = e^{-(L\otimes I_d)t}$, the second term above is bounded as
\begin{align}
\int_0^t \|\Phi(t-s)\gamma_s\| & = \int_0^t \|e^{-(L\otimes I_d)(t-s)}\|\gamma_s\\
& \leq C\int_0^t e^{-\lambda_2(t-s)}s^{-\tau_\gamma}\ds,
\end{align}
for some $C>0$,
where $\lambda_2>0$ is the second smallest eigenvalue of $L$.
%, and where we note that $(t-s)$ is always positive over the domain of the integral.
Since $\tau_\gamma > 0$, this converges to zero as $t\to\infty$.
\end{proof}

We now prove Theorem \ref{thrm:cont-conv-cp}.
\begin{proof}[Proof (Theorem \ref{thrm:cont-conv-cp})]
Part (i) of the theorem follows from Lemma \ref{lemma:cont-consensus}. We now prove part (ii) of the theorem.
Let $S(\tau) = \int_{0}^\tau \alpha_r \,dr$ and let $T(t)$ denote the inverse of $S(\tau)$ so that $T(S(\tau)) = \tau$. Letting $\vy_n(t) = \vx_n(T(t))$ we have
\begin{equation} \label{eq:ODE-alt-form1}
\dot \vy_n(t) = \gamma_t\sum_{\ell\in\Omega_n} (\vy_\ell(t) - \vy_n(t)) - \nabla f_n(\vy_n(t)),
\end{equation}
$n=1,\ldots,N$,
where $\gamma_t = \frac{\beta_{T(t)}}{\alpha_{T(t)}} \to \infty$ as $t\to\infty$. Since \eqref{eq:ODE-alt-form1} is equivalent to \eqref{dynamics_CT} up to a time change, we will prove the result for solutions to \eqref{eq:ODE-alt-form1}.

By Lemma \ref{lemma:cont-consensus}, it is sufficient to show that the mean process, $\vyavg(t)$, converges to the set of critical points of $f$.
Noting that $\sum_{n=1}^N\sum_{\ell\in\Omega_n}(\vy_\ell(t) - \vy_n(t)) = 0$ (because $G$ is undirected), the average dynamics may be expressed as
\begin{align}
\dot{\vy}_{\textup{avg}}(t) & = \frac{1}{N}\sum_{n=1}^N \dot \vy_n(t)\\
& = \frac{1}{N}\sum_{n=1}^N \left( \gamma_t\sum_{\ell\in\Omega_n}(\vy_\ell(t) - \vy_n(t)) - \grad f_n(\vy_n(t)) \right) \\
& =  -
\frac{1}{N}\sum_{i=1}^N \bigg( \Big(\grad f_n(\vy_n(t)) - \grad f_n(\vyavg(t)) \Big)\\
& \quad\quad  + \grad f_n(\vyavg(t)) \bigg)
\end{align}
\begin{align}
& = - \frac{1}{N}\sum_{i=1}^N \grad f_n(\vyavg(t)) + \vr(t)\\
& = - \grad f(\vyavg(t)) + \vr(t), \label{eqn:avg-ode}
\end{align}
where $\vr(t) =  -\frac{1}{N}\sum_{i=1}^N \Big(\grad f_n(\vy_n(t)) - \grad f_n(\vyavg(t)) \Big)$.

By Assumptions \ref{a:lip-grad} and \ref{a:coercive} we see that $\vr(t)\to 0$ as $t\to\infty$.
Recalling Definition \ref{def:perturbed-soln}, solutions to \eqref{eqn:avg-ode} may be viewed as perturbed solutions of the ODE
\begin{equation} \label{eqn:avg-autonomous-unperturbed}
\dot \vy = -\grad f(\vy).
\end{equation}
Let $\Lambda$ denote the set of critical points of $f$.
Since $f\in C^2$, Sard's theorem implies that $f(\Lambda)$ has empty interior.
By Theorem \ref{thrm:BPS-limit}, solutions to \eqref{eqn:avg-ode} converge to the critical points set of $f$.
\end{proof}

\section{Nonconvergence to Saddle Points}\label{sec:nonconvergence}

\subsection{Generalized Problem Setup}
It will simplify the presentation and proofs if we consider a slight generalization of the distributed optimization framework. Namely, we will consider the distributed optimization problem as a special case of subspace constrained optimization. To this end, let $M\geq 1$ denote the dimension of the ambient space, let $h:\R^{M}\to\R$ be a $C^2$ function, and let $Q\in \R^{M\times M}$ be a positive semidefinite matrix. Consider the following optimization problem
\begin{align}
\min_{x\in\R^m} & \quad h(x)\\
\text{subject to} &\quad  x\in\calN(Q),
\end{align}
and the following dynamics for addressing this problem
\begin{equation} \label{eq:ODE1}
\dot \vx(t) =  - \nabla h(\vx(t)) -\beta_t Q \vx(t),
\end{equation}
where $\beta_t$ is some pre-specified weight function of class $C^1$ satisfying $\beta_t \to \infty$ as $t\to\infty$.

Note that the dynamics \eqref{eq:ODE1} may be viewed as $\dot \vx(t) = -\nabla_x \left( h(\vx(t)) + \vx^TQ\vx(t)\right)$, i.e., as $\beta_t\to\infty$, $\vx(t)$ is forced towards the constraint set.
%In principle, our hope is that $\vx(t)$ will converge to local optima of $h$ within the constraint set.

Under Assumptions \ref{a:G-connected-undirected}--\ref{a:step-size-CT}, \eqref{dynamics_CT} is a special case of \eqref{eq:ODE1}. To see this, first observe that \eqref{dynamics_CT} (or rather, \eqref{eqn:non-homogenous-dynamics0}) is equivalent to the following ODE after a time change
$$
\dot \vx = -\beta_t(L\otimes I_d) \vx - (\grad f_n(\vx))_{n=1}^N,
$$
where $\beta_t \to \infty$. This fits the template of \eqref{eq:ODE1} where we let $M=Nd$, let $h:\R^{Nd}\to \R$ be given by the sum function\footnote{Note that this differs from \eqref{eqn:f-def} in that we permit the arguments of $f_n$ to differ.} $h(x) = \sum_{n=1}^N f_n(x_n)$, and let $Q = L\otimes I_d$.

Within this generalized framework, we would like to capture Assumption \ref{a:saddle-point}.
%capture the idea that $x$ is a critical point of the sum function \eqref{eqn:f-def}.
To this end, let $\calC = \calN(Q)$; we say that a point $x^*\in \mathcal{C}$ is a critical point of the restricted function $h\vert_{\mathcal{C}}$ if $\nabla h\vert_{\calC}(x^*) = 0$, where $\nabla h\vert_{\calC}(x^*)\in \R^m$ is taken with respect to some orthonormal basis of $\calC$, and $m=\dim\calC$.
%$\mathbf{d}_v h(x^*) = 0$ for all $v\in \mathcal{C}$, where $\mathbf{d}_v h$ denotes the directional derivative of $h$ in the direction $v$.
Let $\nabla^2 h\vert_{\calC}(x^*)\in \R^{m\times m}$ denote the Hessian of $h\vert_{\calC}$ taken with respect to some orthonormal basis of $\calC$. We say that $x^*$ is a nondegenerate saddle point of $h\vert_{\calC}$ if $\det \nabla^2 h\vert_{\calC}(x^*) \not= 0$, and $\nabla^2 h\vert_{\calC}(x^*)$ has at least one positive and one negative eigenvalue.

The following theorem demonstrates the existence of stable manifolds for \eqref{eq:ODE1}.
\begin{theorem} \label{thrm:main-continuous}
Suppose $h\in C^2$, $\beta_t\in C^1$ and $\dim \calN(Q) \geq 2$. Suppose 0 is a nondegenerate saddle point of $h\vert_{\mathcal{C}}$ and let $p$ denote the number of negative eigenvalues of $\nabla^2 h\vert_{\calC}(0)$. Then for all $t_0$ sufficiently large there exists a manifold $S_{t_0}\subset \R^{M}$ with dimension $M-p$ such that the following holds: A solution $\vx$ to \eqref{eq:ODE1} converges to 0 if and only if $\vx$ is initialized on $\calS$, i.e., $\vx(t_0) = x_0\in \calS$.
\end{theorem}
Since, under Assumptions \ref{a:G-connected-undirected}--\ref{a:step-size-CT}, \eqref{dynamics_CT} is a special case of \eqref{eq:ODE1} this implies Theorem \ref{thrm:non-convergence}.

\subsection{Proof of Theorem \ref{thrm:main-continuous}}
\noindent 1. (\textit{Recenter}) By the implicit function theorem, there exists a function $g\in C^1([0,\infty); \R^M)$ such that, for each $\beta \geq 0$, $g(\beta)$ is a critical point of the penalized function $h(x)+\beta x^TQx$ and $g(\beta)\to 0$ as $\beta\to\infty$.

Letting $\vy(t) = \vx(t) - g(\beta_t)$ we see that $\vx$ is a solution to \eqref{eq:ODE1} if and only if $\vy$ is a solution to
\begin{equation}\label{eq:ODE2}
\dot \vy = -\nabla h(\vy + g(\beta_t)) - \beta_t Q(\vy + g(\beta_t)) - g'(\beta_t)\dot \beta_t,
\end{equation}
where $g'$ denotes the vector $(\frac{\partial g_i}{\partial \beta})_{i=1}^M$.
For $t\geq 0$ let
$$A(t) := \nabla^2_x \left( h(x) + \beta_t x^T Q x \right)\big\vert_{x = g(\beta_t)}$$
and let
$F(y,t) := -\nabla h(y + g(\beta_t)) - \beta_t Q(y + g(\beta_t)) - A(t)y$ so that we may express \eqref{eq:ODE2} as
\begin{equation}\label{eq:ode-y}
\dot \vy(t) = A(t)\vy(t) + F(\vy(t),t) - g'(\beta_t)\dot \beta_t.
\end{equation}

2. (\textit{Diagonalize}) For each $t\geq 0$, let $U(t)$ be a unitary matrix that diagonalizes $A(t)$, so that $U(t)A(t)U(t)^T = \Lambda(t)$, where $\Lambda(t)$ is diagonal. Since $\beta_t\in C^1$ we may construct $U(t)$ as a differentiable function of $t$. Changing coordinates again, let $\vz(t) = U(t)\vy(t)$ so that $\vy$ is a solution to \eqref{eq:ode-y} if and only if $\vz$ is a solution to
\begin{align}
\dot \vz(t) = & U(t)\dot \vy(t) + \dot U(t) \vy(t)\\
= & U(t)( A(t)U(t)^T\vz(t) + F(U(t)^T\vz(t),t)\\
& - g'(\beta_t)\dot \beta_t ) + \dot{U}(t) U(t)^T \vz(t)
\end{align}
Letting $\tilde F(z,t) := U(t)F(U(t)^T z,t) + \dot U(t) U(t)z $, the above is equivalent to
\begin{equation}\label{eq:ODE-z}
\dot \vz(t) = \Lambda(t)\vz(t) + \tilde F(\vz(t),t) - U(t)g'(\beta_t)\dot \beta_t.
\end{equation}\label{disp:unif-lip}
Note that $F(0,t) = 0$ and $F(y,t) = o(|y|^2)$ for $t\geq 0$. Consequently, for any $\epsilon>0$ there exists an $r>0$ and $T\geq 0$ such that for all $t\geq T$ and $z,\tilde z\in B_r(0)$ we have
\begin{equation} \label{eq:Lip-F}
|\tilde F(z,t) - \tilde F(\tilde z,t)| \leq \e |z-\tilde z|.
\end{equation}

\bigskip
\noindent 3. (\textit{Compute Stable Solutions}) Let $\lambda_1(t),\ldots,\lambda_M(t)$ denote the eigenvalues of $\Lambda(t)$. We may assume the eigenvalues are ordered so each $\lambda_i(t)$ varies smoothly in $t$. For $T$ sufficiently large, the sign of $\lambda_i(t)$ remains constant for all $t \geq T$, for each $i$. Without loss of generality assume that the first $k<M$ diagonal entries (eigenvalues) of $\Lambda(t)$ are negative and the remaining diagonal entries are positive for all $t$ sufficiently large. Let $\Lambda(t)$ be decomposed as
$$
\Lambda(t) =
\begin{pmatrix}
  \Lambda^s(t) & 0 \\
  0 & \Lambda^u(t)
\end{pmatrix}
$$
where $\Lambda^s(t)\in \R^{k\times k}$ denotes the `stable' diagonal submatrix and $\Lambda^u(t) \in \R^{(M-k)\times (M-k)}$ denotes the `unstable' diagonal submatrix.
Let
\begin{equation} \label{eq:V-def}
V^s(t_2,t_1) :=
\begin{pmatrix}
e^{\int_{t_1}^{t_2} \Lambda^s(\tau)\,d\tau} & 0\\
0 & 0\\
\end{pmatrix},
\end{equation}
\begin{equation}
V^u(t_2,t_1) :=
\begin{pmatrix}
0 & 0\\
0 & e^{\int_{t_1}^{t_2} \Lambda^u(\tau)\,d\tau}\\
\end{pmatrix}.
\end{equation}
By construction we have $\limsup_{t\to\infty} \lambda_j(t) < 0$,  $j=1,\ldots,k$. Hence, we may choose an $\alpha > 0$ such that $\lambda_j(t) < -\alpha < 0$ for $j=1,\ldots, k$ and all $t$ sufficiently large.
We may also choose constants $\sigma> 0$ and $K>0$ such that the following estimates hold
\begin{align}
\label{eq:ex-estimate-s}
\|V^s(t_2,t_1)\| & \leq Ke^{-(\alpha+\sigma)(t_2-t_1)}, \quad \quad t_2\geq t_1\\
\label{eq:ex-estimate-u}
\|V^u(t_2,t_1)\| & \leq Ke^{\sigma (t_2-t_1)}, \quad\quad\quad~~~ t_2\leq t_1.
\end{align}
where $t_1,t_2 \geq t_0$.
Now, suppose $a^s\in \R^k$ and consider the integral equation
\begin{align} \label{eq:integral-eq0}
& \vu(t,a^s) = V^s(t,t_0)
\begin{pmatrix}
a^s\\
0
\end{pmatrix}\\
& + \int_{t_0}^t V^s(t,\tau)\left(\tilde F(\vu(\tau,a^s),\tau) -U(\tau)g'(\beta_\tau)\dot\beta_\tau \right)\,d\tau\\
& - \int_{t}^\infty V^u(t,\tau)\left( \tilde F(\vu(\tau,a^s),\tau) -U(\tau)g'(\beta_\tau)\dot\beta_\tau  \right)\,d\tau,
\end{align}
where $\vu:[t_0,\infty)\times \R^k\to\R^M$. Note that if $t\mapsto \vu(t,a^s)$ is continuous and solves \eqref{eq:integral-eq0} then, $\vu(t,a^s)$ is differentiable and solves \eqref{eq:ODE-z} with componentwise initialization $\vu_i(t_0,a^s) = a^s_i$ for $i=1,\ldots,k$. This may be verified using the variation of parameters formula \cite{lakshmikantham1969differential}.
%(Put more explanation here. This follows from variation of constants, etc.)

Given $t_0\geq 0$, let
\begin{align}\label{eq:c-def}
c(t) := & \int_{t_0}^t V^s(t,\tau) U(\tau) g'(\beta_\tau)\dot\beta_\tau\dtau\\
& +\int_{t}^\infty V^u(t,\tau) U(\tau)g'(\beta_\tau)\dot\beta_\tau\dtau, \quad t\geq t_0.
\end{align}
We remark that $c(t)$ is finite for all $t\geq t_0$ and for any $\eta>0$ we may choose $t_0$ sufficiently large so that $|c(t)|<\eta$ for all $t\geq t_0$.

Suppose $\e < \sigma/6K$ and let $r$ and $T$ be chosen so that \eqref{eq:Lip-F} holds for all $t\geq T$ and  $|c(t)| \leq r/3$ for all $t\geq T$. By Lemma \ref{lemma:contraction}, if $|a_s|< r/3$ and $t_0\geq T$, then the right-hand side of \eqref{eq:integral-eq0} is a contraction on the space
\begin{align}
X_{t_0,a^s} :=\{\vx\in C([t_0,\infty); R^M),~ \vx_i(t_0) = a_i^s+c_i(t_0),\, \\
i=1,\ldots,k, \|\vx\|_{\infty} < \infty\},
\end{align}
equipped with norm $\|\cdot\|_\infty$, where $c(t)$ is defined in \eqref{eq:c-def}. Since this space is complete, there exists a unique $\vu(\cdot,a^s)\in X_{t_0,a^s}$ solving \eqref{eq:integral-eq0}.

\bigskip
\noindent 4. (\textit{Construct Stable Manifold})
We now construct the stable set $\calS$ corresponding to the ODE \eqref{eq:ODE-z}. Let $t_0 \geq T$.
For each $z_0^s\in B_{\frac{r}{3}}(0)\subset \R^k$ let $\vu(\cdot,z_0^s)$ be the (unique) solution to \eqref{eq:integral-eq0} in $X_{T,z_0^s}$. For each $t\in[t_0,\infty)$ define the component map $\psi_j:\R\times \R^k\to \R$ by
$$
\psi_j(t,z_0^s) := \vu_j(t,z_0^s), \quad j=k+1,\ldots,M,$$
and let $\psi = (\psi_j)_{j=k+1}^M$.
The stable manifold (with respect to \eqref{eq:ODE-z}) is given by
$$
\calS :=\{(t,z_0^s,\psi(t,z_0^s)), t\geq T, \, z_0^s\in \R^k\cap B_{\frac{r}{3}}(0)\}.
$$
By construction, for any initialization $(t_0,z_0^s,z_0^u)\in \calS$, the corresponding solution $\vz$ of \eqref{eq:ODE-z} with $\vz(t_0) = (z_0^s,z_0^u)$ satisfies $\vz(t)\to 0$. Moreover, by Lemma \ref{lemma:stable-in-S} we see that $\calS$ contains all stable initializations $(t_0,z_0)$. That is, if $\vz$ is a solution to \eqref{eq:ODE-z} with $\vz(t_0) = z_0$ and $\vz(t)\to 0$, then $(t_0,z_0)\in \calS$.

Having constructed $\calS$ (the stable manifold for \eqref{eq:ODE-z}) the stable manifold  for \eqref{eq:ODE1}, denoted here by $\tilde \calS$, is obtained by an appropriate change of coordinates, $\tilde \calS :=\{(t,x)\in \R\times\R^M:~ U(t)(x-g(\beta_t)) \in \calS \}.$

\section{Conclusion} \label{sec:conclusion}
We have considered the distributed gradient descent dynamics \eqref{dynamics_CT} for nonconvex optimization. We showed that the dynamics converge to the set of critical points of the nonconvex objective (Theorem \ref{thrm:cont-conv-cp}). Furthermore, the dynamics may only converge to a saddle point of the objective if initialized from some special low-dimensional stable manifold.

\section*{Appendix}
This appendix contains some intermediate results required for the proof of Theorem \ref{thrm:main-continuous}.

The following lemma shows that the right-hand side of \eqref{eq:integral-eq0} is a contraction. Before presenting the lemma, we define a few useful quantities.
Given $a_s\in \R^k$, let $\calT: X_{t_0,a^s}\to X_{t_0,a^s}$ be given by
\begin{align}
\calT(\vu)(t) := &
 V^s(t,t_0)
\begin{pmatrix}
a^s\\
0
\end{pmatrix}
+ \int_{t_0}^t V^s(t,\tau)\tilde F(\vu(\tau))\,d\tau\\
& - \int_{t}^\infty V^u(t,\tau) \tilde F(\vu(\tau))\,d\tau + c(t),
\end{align}
where, for convenience, we suppress the argument $a^s$ previously used in $\vu$.
\begin{lemma} [$\calT$ is a contraction] \label{lemma:contraction}
Let $\sigma$, $\alpha$, and $K$ be chosen so that \eqref{eq:ex-estimate-s} is satisfied. Let $0<\epsilon <\frac{\sigma}{6K}$, and let $r$ and $T$ be chosen so that \eqref{eq:Lip-F} holds and $|c(t)|\leq r/3$ holds for all $t\geq T$. Let $a^s\in \R^k$ with $|a^s|< \frac{r}{3}$. Then $\calT:X_{t_0,a^s}\to X_{t_0,a^s}$ is a contraction.
\end{lemma}
\begin{proof}
First, claim that if $\vu\in X_{t_0,a^s}$ and $\|\vu\|_{\infty} \leq r$, then $\|\calT(\vu)\|_{\infty} \leq r$. To see this, note that
\begin{align}
|\calT(\vu)(t)|  \leq & e^{-\alpha (t-t_0)}|a^s| + |c(t)|\\
& + \int_{t_0}^t Ke^{-(\alpha+\sigma)(t-\tau)}\e |\vu(\tau)|\dtau\\
& + \int_{t}^\infty Ke^{\sigma(t-\tau)}\e |\vu(\tau)|\dtau\\
 \leq & e^{-\alpha (t-t_0)}|a^s| + |c(t)|\\
 & + K\e r\int_{t_0}^t e^{-\sigma(t-\tau)}\dtau + K\e r\int_{t}^\infty e^{\sigma(t-\tau)}\dtau\\
 \leq & e^{-\alpha (t-t_0)}|a^s| + |c(t)| + \frac{2K\e r}{\sigma}\\
 \leq & r/3 + r/3 + r/3,
\end{align}
where in the last line we use the assumptions made on $|a^s|$, $\e$, and $t_0$ in the statement of the lemma.

Suppose now that $\vu,\hat \vu\in X_{t_0,a_s}$, with $\|\vu\|_{\infty},\|\hat \vu\|_{\infty} \leq r$. Let $M = \|\vu - \hat \vu\|_\infty$. For $t\geq t_0$ we have
\begin{align}
|\calT(\vu)(t) - \calT(\hat \vu)(t)| \leq & \int_{t_0}^t Ke^{-(\alpha - \sigma)(t-\tau)}\e|\vu(\tau) - \hat \vu(\tau)|\dtau\\
& + \int_{t}^\infty K e^{\sigma(t-\tau)}\e |\vu(\tau) - \hat \vu(\tau)|\dtau\\
\leq & \e KM\int_{t_0}^t e^{-(\alpha - \sigma)(t-\tau)}\dtau\\
& + \e K M\int_{t}^\infty e^{\sigma(t-\tau)}\dtau\\
\leq & \frac{2\e K}{\sigma}M.
\end{align}
Given our choice of $\e$ we have $\frac{2\e K}{\sigma} < 1$, hence, $\calT$ is a contraction.
\end{proof}

\begin{lemma} [$\calS$ contains all stable initializations] \label{lemma:stable-in-S}
Let $\e$, $r$, and $T$ be chosen as in the construction of $\calS$. Let $a^s\in \R^k$, with $|a^s|< r/3$, let $t_0\geq T$ and suppose that $\vz$ is a solution to \eqref{eq:ODE-z} with $\vz_i(t_0) = z_0 = (z_0^s, z_0^u)$ and $z_0^s = a^s$. If $\vz(t)\to 0$ as $t\to\infty$ then $(t_0,z_0)\in \calS$.
\end{lemma}
\begin{proof}
By variation of constants we see that
\begin{align} \label{eq:integral-eq}
\vz(t) := & V^s(t,t_0)\vz(t_0) + V^u(t,t_0)c\\
& +  \int_{t_0}^t V(t,\tau)\left(\tilde F(\vz(\tau),\tau)  -U(\tau)g'(\tau)\dot\beta_\tau \right)\,d\tau\\
& - \int_{t}^\infty V^u(t,\tau)\left(\tilde F(\vz(\tau)) -U(\tau)g'(\tau)\dot\beta_\tau \right)\,d\tau,
\end{align}
where $c = \vz(t_0) + \int_{t_0}^\infty V^u(t_0,\tau)\left( \tilde F(\vz(\tau)) -U(\tau)g'(\tau)\dot\beta_\tau \right)\dtau$. Note that integral in $c$ converges by \eqref{eq:V-def} and the fact that $\int_{t_0}^\infty U(\tau)g'(\tau)\beta_\tau\dtau < \infty$.
Every term on the right hand side of \eqref{eq:integral-eq} is uniformly bounded in $t$, except possibly the term $V^u(t,t_0)c$. In particular, if $c_j\not = 0$, $j> k$, then $|V^u(t,t_0)c|\to\infty$.
Since the left hand side of \eqref{eq:integral-eq} is bounded, it follows that the right hand side is bounded and thus all $c_j$, $j>k$ must be zero and hence $V^u(t,t_0)c = 0$.

This implies that $\vu(\cdot,a^s) = \vz$ is a solution to the integral equation \eqref{eq:integral-eq0} given $a^s$. By Lemma \ref{lemma:contraction} we see that $\vu(t,a^s)$
is the unique continuous solution of \eqref{eq:integral-eq0} given $a^s$. By the definitions of $\calS$ and $\psi$ we thus see that $(t_0,z_0)\in \calS$.
\end{proof}

\bibliographystyle{IEEEtran}
\bibliography{myRefs}

% Generated by IEEEtran.bst, version: 1.14 (2015/08/26)
\begin{thebibliography}{10}
\providecommand{\url}[1]{#1}
\csname url@samestyle\endcsname
\providecommand{\newblock}{\relax}
\providecommand{\bibinfo}[2]{#2}
\providecommand{\BIBentrySTDinterwordspacing}{\spaceskip=0pt\relax}
\providecommand{\BIBentryALTinterwordstretchfactor}{4}
\providecommand{\BIBentryALTinterwordspacing}{\spaceskip=\fontdimen2\font plus
\BIBentryALTinterwordstretchfactor\fontdimen3\font minus
  \fontdimen4\font\relax}
\providecommand{\BIBforeignlanguage}[2]{{%
\expandafter\ifx\csname l@#1\endcsname\relax
\typeout{** WARNING: IEEEtran.bst: No hyphenation pattern has been}%
\typeout{** loaded for the language `#1'. Using the pattern for}%
\typeout{** the default language instead.}%
\else
\language=\csname l@#1\endcsname
\fi
#2}}
\providecommand{\BIBdecl}{\relax}
\BIBdecl

\bibitem{rabbat2004distributed}
M.~Rabbat and R.~Nowak, ``Distributed optimization in sensor networks,'' in
  \emph{Proceedings of the 3rd International Symposium on Information
  Processing in Sensor Networks}, 2004, pp. 20--27.

\bibitem{chen2012diffusion}
J.~Chen and A.~H. Sayed, ``Diffusion adaptation strategies for distributed
  optimization and learning over networks,'' \emph{IEEE Transactions on Signal
  Processing}, vol.~60, no.~8, pp. 4289--4305, 2012.

\bibitem{cao2012overview}
Y.~Cao, W.~Yu, W.~Ren, and G.~Chen, ``An overview of recent progress in the
  study of distributed multi-agent coordination,'' \emph{IEEE Transactions on
  Industrial Informatics}, vol.~9, no.~1, pp. 427--438, 2012.

\bibitem{boyd2011distributed}
S.~Boyd, N.~Parikh, E.~Chu, B.~Peleato, and J.~Eckstein, ``Distributed
  optimization and statistical learning via the alternating direction method of
  multipliers,'' \emph{Foundations and Trends{\textregistered} in Machine
  learning}, vol.~3, no.~1, pp. 1--122, 2011.

\bibitem{kar2019clustering}
S.~Kar and B.~Swenson, ``Clustering with distributed data,'' 2019, submitted
  for publication. Online: https://arxiv.org/abs/1901.00214.

\bibitem{nedic2009distributed}
A.~Nedi{\'c} and A.~Ozdaglar, ``Distributed subgradient methods for multi-agent
  optimization,'' \emph{IEEE Transactions on Automatic Control}, vol.~54,
  no.~1, p.~48, 2009.

\bibitem{lee2018distributed}
C.~Lee, C.~H. Lim, and S.~J. Wright, ``A distributed quasi-newton algorithm for
  empirical risk minimization with nonsmooth regularization,'' in
  \emph{Proceedings of the 24th ACM SIGKDD International Conference on
  Knowledge Discovery \& Data Mining}, 2018, pp. 1646--1655.

\bibitem{di2015distributed}
P.~Di~Lorenzo and G.~Scutari, ``Distributed nonconvex optimization over
  networks,'' in \emph{Proceedings of the 6th International Workshop on
  Computational Advances in Multi-Sensor Adaptive Processing (CAMSAP)}, 2015,
  pp. 229--232.

\bibitem{sun2016distributed}
Y.~Sun, G.~Scutari, and D.~Palomar, ``Distributed nonconvex multiagent
  optimization over time-varying networks,'' in \emph{2016 50th Asilomar
  Conference on Signals, Systems and Computers}.\hskip 1em plus 0.5em minus
  0.4em\relax IEEE, 2016, pp. 788--794.

\bibitem{welikala2019distributed}
S.~Welikala and C.~G. Cassandras, ``Distributed non-convex optimization of
  multi-agent systems using boosting functions to escape local optima,''
  \emph{arXiv preprint arXiv:1903.04133}, 2019.

\bibitem{bianchi2012convergence}
P.~Bianchi and J.~Jakubowicz, ``Convergence of a multi-agent projected
  stochastic gradient algorithm for non-convex optimization,'' \emph{IEEE
  Transactions on Automatic Control}, vol.~58, no.~2, pp. 391--405, 2012.

\bibitem{di2016next}
P.~Di~Lorenzo and G.~Scutari, ``Next: In-network nonconvex optimization,''
  \emph{IEEE Transactions on Signal and Information Processing over Networks},
  vol.~2, no.~2, pp. 120--136, 2016.

\bibitem{dauphin2014identifying}
Y.~N. Dauphin, R.~Pascanu, C.~Gulcehre, K.~Cho, S.~Ganguli, and Y.~Bengio,
  ``Identifying and attacking the saddle point problem in high-dimensional
  non-convex optimization,'' in \emph{Advances in neural information processing
  systems}, 2014, pp. 2933--2941.

\bibitem{lee2016gradient}
J.~D. Lee, M.~Simchowitz, M.~I. Jordan, and B.~Recht, ``Gradient descent only
  converges to minimizers,'' in \emph{Conference on learning theory}, 2016, pp.
  1246--1257.

\bibitem{lee2017first}
J.~D. Lee, I.~Panageas, G.~Piliouras, M.~Simchowitz, M.~I. Jordan, and
  B.~Recht, ``First-order methods almost always avoid strict saddle points,''
  \emph{Mathematical Programming}, pp. 1--27, 2019.

\bibitem{murray2019revisiting}
R.~Murray, B.~Swenson, and S.~Kar, ``Revisiting normalized gradient descent:
  Fast evasion of saddle points,'' \emph{IEEE Transactions on Automatic
  Control}, vol.~PP, pp. 1--1, 2019.

\bibitem{shub2013global}
M.~Shub, \emph{Global stability of dynamical systems}.\hskip 1em plus 0.5em
  minus 0.4em\relax Springer Science \& Business Media, 2013.

\bibitem{coddington1955theory}
E.~A. Coddington and N.~Levinson, \emph{Theory of ordinary differential
  equations}.\hskip 1em plus 0.5em minus 0.4em\relax Tata McGraw-Hill
  Education, 1955.

\bibitem{swenson2019CDC}
B.~Swenson, S.~Kar, H.~V. Poor, and J.~M.~F. Moura, ``Annealing for distributed
  global optimization,'' 2019, to appear in Proceedings of IEEE Conference on
  Decision and Control.

\bibitem{tatarenko2017non}
T.~Tatarenko and B.~Touri, ``Non-convex distributed optimization,'' \emph{IEEE
  Transactions on Automatic Control}, vol.~62, no.~8, pp. 3744--3757, 2017.

\bibitem{daneshmand2018second}
A.~Daneshmand, G.~Scutari, and V.~Kungurtsev, ``Second-order guarantees of
  distributed gradient algorithms,'' \emph{arXiv preprint arXiv:1809.08694},
  2018.

\bibitem{kar2012distributed}
S.~Kar, J.~M. Moura, and K.~Ramanan, ``Distributed parameter estimation in
  sensor networks: Nonlinear observation models and imperfect communication,''
  \emph{IEEE Transactions on Information Theory}, vol.~58, no.~6, pp.
  3575--3605, 2012.

\bibitem{nedic2010constrained}
A.~Nedi{\'c}, A.~Ozdaglar, and P.~A. Parrilo, ``Constrained consensus and
  optimization in multi-agent networks,'' \emph{IEEE Transactions on Automatic
  Control}, vol.~55, no.~4, pp. 922--938, 2010.

\bibitem{nedic2014distributed}
A.~Nedi{\'c} and A.~Olshevsky, ``Distributed optimization over time-varying
  directed graphs,'' \emph{IEEE Transactions on Automatic Control}, vol.~60,
  no.~3, pp. 601--615, 2014.

\bibitem{hong2018gradient}
M.~Hong, J.~D. Lee, and M.~Razaviyayn, ``Gradient primal-dual algorithm
  converges to second-order stationary solutions for nonconvex distributed
  optimization,'' in \emph{Proceedings of the 35th International Conference on
  Machine Learning, PMLR 80}, 2018.

\bibitem{chung1997spectral}
F.~R. Chung and F.~C. Graham, \emph{Spectral graph theory}.\hskip 1em plus
  0.5em minus 0.4em\relax American Mathematical Soc., 1997, no.~92.

\bibitem{benaim2005stochastic}
M.~Bena{\"\i}m, J.~Hofbauer, and S.~Sorin, ``Stochastic approximations and
  differential inclusions,'' \emph{SIAM Journal on Control and Optimization},
  vol.~44, no.~1, pp. 328--348, 2005.

\bibitem{lakshmikantham1969differential}
V.~Lakshmikantham and S.~Leela, \emph{Differential and Integral Inequalities:
  Theory and Applications: Volume I: Ordinary Differential Equations}.\hskip
  1em plus 0.5em minus 0.4em\relax Academic press, 1969.

\end{thebibliography}

\end{document}